%% file: main.tex
\documentclass[11pt]{amsart}
\usepackage[left=1.0in,top=0.95in,right=1.0in,bottom=0.95in]{geometry}

\usepackage{graphicx}
\usepackage{amsmath,amsthm,amssymb,amsfonts,amscd,enumitem,booktabs,float}
\usepackage[dvipsnames]{xcolor}
\usepackage{tikz}
\usepackage{caption}
\usepackage[hidelinks]{hyperref}
\usetikzlibrary{calc}
\def\edgewidth{0.83333pt}
\def\vertexsize{2pt}

\newcommand{\vertex}[2][1]{%
  \fill (#2) circle [radius=#1*\vertexsize];
}
\newtheorem{Th}{Theorem}[section]
\newtheorem{Lemma}[Th]{Lemma}
\newtheorem{Cor}[Th]{Corollary}
\newtheorem{Def}[Th]{Definition}
\newtheorem{Prop}[Th]{Proposition}
\theoremstyle{definition}

\newtheorem{Ex}[Th]{Example}

\newcounter{commentlabel}
\newlength{\poincarecommentlift}
\makeatletter
\DeclareRobustCommand{\COMMENT}[1]{\@bsphack%
	\stepcounter{commentlabel}%
	\vbox to0pt{%
		\setlength{\fboxrule}{0.75pt}%
		\setlength{\fboxsep}{0.75pt}%
		\setlength{\poincarecommentlift}{1ex}%
		\addtolength{\poincarecommentlift}{\fboxrule}%
		\addtolength{\poincarecommentlift}{\fboxsep}%
		\vss\color{red}%
		\rlap{\rlap{\vrule height\poincarecommentlift width\fboxrule}\raise \poincarecommentlift%
		\hbox{\fcolorbox{red}{yellow}{%
\normalfont\footnotesize\ttfamily\bfseries\thecommentlabel}}}}%
	\marginpar{\noindent\raggedright%
	\textbf{\color{red}\thecommentlabel}:\thinspace\footnotesize#1}%
}

\title{Divisors and Harmonic Morphisms on Metric Graphs of Pseudocompact Type}

\author{Benjamin Mascuch}
\author{Montserrat Teixidor i Bigas}

\address{Benjamin Mascuch \\ University of Cambridge, Cambridge, UK}
\email{bam62@cam.ac.uk}
\address{Montserrat Teixidor i Bigas \\ Tufts University, Medford, USA}
\email{mteixido@tufts.edu}

\date{}

\begin{document}
\thispagestyle{empty}
\vspace*{-.9cm}
\begin{abstract}
A metric graph is of pseudocompact type if identifying parallel edges produces a tree. We give a constructive proof that, for such graphs, divisorial $d$-gonality is equivalent to the existence of a degree $d$ harmonic morphism to a tree. 
This mirrors the algebraic correspondence, for curves of compact type, between limit linear series of dimension one and admissible covers. We also deduce lifting results for positive-rank divisors, with a genus-preserving refinement when identifying parallel edges produces a path.
Finally, we study the Brill--Noether theory in the path case.
\end{abstract}
\maketitle

\vspace*{-0.9cm}

\section{Introduction}

A recurring theme in tropical geometry is that a classical algebro-geometric statement which fails for its natural tropical analogue is often, in some sense, ``almost true.'' What one first writes down tropically may be close to the right statement and, by slightly adjusting the tropical notion or restricting to a sufficiently well-behaved class of objects, one can recover a true statement that closely resembles the classical one.

We give two illustrative examples. Unlike the classical case, a tropical curve is not in general determined by its principally polarized Jacobian, but Caporaso and Viviani characterize the fibers of the tropical Torelli map and recover a strong Torelli theorem for $3$-connected graphs \cite{CV}. Likewise, in Brill--Noether theory, the dimension of a tropical Brill--Noether locus need not agree with its classical counterpart, but the expected dimension can be recovered by replacing dimension with an extended notion of rank \cite{LPP}.

The motivating example for this paper is gonality. For a smooth projective curve $C$, its gonality is both the least degree of a divisor of rank at least $1$ and the least degree of a nonconstant morphism $C\to\mathbb P^1$. Indeed, a basepoint-free $g^1_d$ determines such a morphism, while the fibers of a degree-$d$ morphism form a basepoint-free $g^1_d$. Since a positive-rank divisor of minimal degree cannot have a base point, these two definitions give the same invariant.

Both descriptions of gonality have natural tropical analogues. The divisorial gonality of a metric graph is the least degree of a divisor of rank at least $1$, while its geometric gonality is the least degree of a finite harmonic morphism from a tropical modification to a metric tree. The fiber of such a morphism is a divisor of rank at least $1$, so geometric gonality is at least divisorial gonality. The converse fails. In an example attributed to Luo in \cite[Example~5.13]{ABBR2}, a metric graph carries a basepoint-free divisor of degree $3$ and rank $1$ but does not admit a finite degree-$3$ harmonic morphism to a tree, even after tropical modification.

In low degree, however, the equivalence is much better behaved. Chan proved that the two notions agree in degree $2$ \cite{C}. Melo and Zheng established the analogous result in degree $3$ for $3$-edge-connected metric graphs \cite{MZ1}, and later extended it to graphs of lower edge connectivity, apart from necklace graphs \cite{MZ2}. Thus, the trigonal equivalence can be recovered by excluding a particular family of combinatorial types. This suggests that, although gonality itself may depend on the edge lengths, the obstruction to the equivalence is governed by the underlying graph. It is therefore natural to seek combinatorial types for which the two notions agree for every choice of metric.

A natural place to begin this search is with chains of loops. Their divisor theory has an especially explicit combinatorial description, which helps explain why they have appeared so often in tropical Brill--Noether theory, beginning with the tropical proof of the Brill--Noether theorem \cite{CDPR} and continuing through work on special divisors, lifting, and limit linear series \cite{Pfl,CJP,LT}. Replacing the two edges of each loop in a chain of loops by an arbitrary number of edges (a banana graph) gives what we call a metric banana path. More generally, arranging banana graphs along an arbitrary tree gives the class that we call metric graphs of pseudocompact type. These are the metric analogues of the banana paths and banana trees recently studied in \cite{BCD}, where the authors suggested considering their metric counterparts.

We show that for these graphs the two notions of gonality always agree.

\begin{Th}
On every metric graph of pseudocompact type, divisorial and geometric gonality coincide.
\end{Th}

Unlike the low-degree results above, this statement is not restricted to a fixed degree. Indeed, metric graphs of pseudocompact type can have arbitrarily large gonality.

We actually prove a stronger statement that holds at the level of individual positive-rank divisors and more closely parallels the algebraic situation. To state it, we introduce tree-induced and path-induced divisors. A divisor $D$ on a metric graph $\Gamma$ is tree-induced if it is linearly equivalent to the retraction of a fiber of a non-degenerate harmonic morphism from a tropical modification of $\Gamma$ to a metric tree. It is path-induced if the target can be chosen to be a path. Every tree-induced divisor has positive rank, while the converse fails as alluded to above.

For metric graphs of pseudocompact type, this obstruction disappears.

\begin{Th}\label{thm:pseudocompact-tree-induced}
On metric graphs of pseudocompact type, every positive-rank divisor is tree-induced.
\end{Th}

The proof of this theorem is constructive. Let $\Gamma$ be a metric graph of pseudocompact type and let $D$ be a positive-rank divisor on $\Gamma$. The simplification of the canonical loopless model of $\Gamma$ is a tree $T$. We compare the reduced representatives of $D$ at adjacent vertices of $T$ by rational functions and combine these functions into a single function on $\Gamma$. This function naturally determines a metric quotient of $T$ and a map from $\Gamma$ to this metric tree. The same reduced representatives of $D$ also measure the failure of this map to be harmonic. We repair this failure by attaching the necessary subtrees of the target to the source, obtaining a suitable harmonic morphism whose fibers retract to divisors linearly equivalent to $D$.

Before proving this theorem, we study tree-induced and path-induced divisors on arbitrary metric graphs. We show that a divisor on a metric graph $\Gamma$ is tree-induced if and only if it is linearly equivalent to the tropicalization of a positive-rank divisor on a smooth proper curve having $\Gamma$ as a skeleton. In the path-induced case, the curve may be chosen to have the same genus as $\Gamma$. We also show that being tree-induced is preserved under the addition of effective divisors and give some sufficient numerical criteria. In particular, we show that every positive-rank divisor on a hyperelliptic metric graph is tree-induced.

Combining Theorem~\ref{thm:pseudocompact-tree-induced} with the above characterization of being tree-induced gives a lifting result for metric graphs of pseudocompact type. Every positive-rank divisor on such a graph lifts, up to linear equivalence, to a positive-rank divisor on a smooth proper curve having the graph as a skeleton. For metric banana paths, the lift can moreover be chosen genus-preserving. See Corollary~\ref{cor:pseudocompact-lifting} for the precise statement.

There is a closely related classical picture for degenerating algebraic curves. Harris and Mumford introduced admissible covers deforming degree-$d$ maps to $\mathbb P^1$ in order to study the closure of the $d$-gonal divisor in $\overline{\mathcal M}_g$ \cite{HMu}. Eisenbud and Harris later introduced limit linear series on curves of compact type, generalizing linear series on smooth curves to carry out Brill--Noether arguments by degeneration \cite{EH}. These two theories encode different aspects of the same degeneration. Admissible covers retain the map itself, while limit linear series retain compatible linear series on the irreducible components. We give an explicit construction of an admissible cover from any limit linear series of dimension $1$ on a curve of compact type after passing to a semistably equivalent source, demonstrating the parallel with our tropical results.

We also study the Brill--Noether theory and more specifically the gonality of metric banana paths. Chains of loops have been useful not only in the study of Brill--Noether theory for general curves, but also in the study of Hurwitz spaces. However, their Clifford index is completely determined by their gonality \cite{Co}, so they may not be able to capture the behavior of curves of higher dimension Clifford indices. It is therefore natural to look for a larger class of graphs whose divisor theory remains explicit. Metric banana paths may provide one such class. We prove that a metric banana path cannot be Brill--Noether general if any of its bananas has genus at least three. On the other hand, chains of genus-two bananas behave much more like chains of loops. In particular, we show that, for generic edge lengths, they have the same gonality as a general curve of their genus. We also construct loci of prescribed gonality in the expected codimension.

The paper is organized as follows. In Section~\ref{sec:algebraic-case}, we relate limit linear series on curves of compact type to admissible covers. In Section~\ref{sec:tropical-prelims}, we review the tropical background needed throughout the paper. In Section~\ref{sec:tree-induced-divisors}, we develop the theory of tree-induced and path-induced divisors on arbitrary metric graphs. In Section~\ref{sec:pseudocompact-type}, we prove the main theorem for metric graphs of pseudocompact type and derive its consequences for gonality and lifting. In Section~\ref{sec:MBP-BN-theory}, we study Brill--Noether theory on metric banana paths.

\section{The algebraic case for curves of compact type}\label{sec:algebraic-case}

As mentioned in the introduction, admissible coverings were introduced in  \cite{HMu} in a ad-hoc basis to compactify the locus of curves of odd genus $g$ 
having a $g^1_{\frac{g+1}2}$.
In particular, the target curve was taken to be rational and only codimension one subloci of $\overline{{\mathcal M}^1_{g,d}}$  had to be considered.
A  general definition of admissible covering appears in \cite{HMo} Def 3-149 and reads as follows (higher order ramification at the point $p_i$ could be easily added.):
\begin{Def}\label{def:admiscov} Given a stable pointed curve $(B,p_1,\dots,p_n)$, an  admissible cover is a stable curve $C$ together with a regular map $f:C\to B$ such that 
\begin{itemize}
\item The inverse image of the non-singular points of $B$  is the set of non-singular points in $C$.
\item The map is simply ramified over the points $p_i$.
\item The inverse image of the nodes in $B$ are the nodes in $C$ and the order of ramification of the two branches of a given node match:
 that is if  locally nodes in $B, C$ are given by $x_by_b=0, x_cy_c=0$ respectively, then $f^*(x_b)=x_c^m, f^*(y_b)=y_c^m$ for the same value of $m$ (depending on the node).
 \end{itemize} 
\end{Def}

A nodal curve is said to be of {\bf compact type} if its dual graph is a tree or equivalently its Jacobian is compact. 
Let $C$ be a curve of compact type, $C=\cup C_i$ where the $C_i$ are irreducible components.
List  the points on $C_i$ $P_{i,i_1}, \dots, P_{i, i_{k_i}}$ so that $C$ is obtained from the $C_i$ by identifying   $P_{i,i_t}\in C_i$  with  $P_{i_t,i}\in C_{i_t}$.
Limit linear series on curves of compact type were defined by Eisenbud and Harris (see \cite{EH}).

\begin{Def}\label{def:lls} Given a curve of compact type, $C=\cup C_i$ as above, a limit linear series of degree $d$ and dimension $r$  is the data of 
\begin{itemize}
\item A line  bundle $L_i$ of degree $d$ on each component $C_i$.
\item A space of sections $V_i$ of dimension $r+1$  for each $L_i$.
 \end{itemize} 
 satisfying the following condition: if $a^0_{i,i_t},\dots, a^r_{i,i_t}$ are the distinct orders of vanishing of the sections in $V_i$ at the node $P_{i,i_t}$, then $a^k_{i,i_t}+a^{r-k}_{i_t,i}\ge d$.
 When this inequality is an equality at every node and for every $k$, then the limit linear series is said to be refined.
\end{Def}

Recall that a connected  nodal curve is stable if  every rational component has at least three nodes.
From a semistable curve, one can obtain a stable curve of the same genus by contracting rational components with one or two nodes.
Conversely, from a stable curve, one can obtain a semistable one by adding a rational component to separate the two points that were identified at a node 
or by adding a rational curve attached to the rest at only one point.
Curves obtained from each other in this way are said to be semistably equivalent.

\begin{Lemma}\label{lem:refine-pencil}
Given a limit linear series on a curve of compact type $C$, there is a semistably equivalent curve $C'$  again of compact type, on which the induced limit linear series is refined with no fixed points away from the nodes.
\end{Lemma}
\begin{proof} If the limit linear series fails to be refined at some node, it is a standard procedure  to blow up that node to make it refined. 
Similarly, if there are fixed points, one can attach rational components at these points with suitable ramification order so that the old fixed points become nodes.
\end{proof}

Our goal is to show the following: 
\begin{Th} Let  $C=\cup_{i\in I} C_i$ be a curve of compact type.
Let  $(L_i, V_i)_{i\in I}$ be  a  limit linear series of dimension one on $C$.
There is then a curve $\widetilde{C}$ semistably equivalent to $C$ 
and an admissible cover from $\widetilde{C}$ to a tree  that induces the limit linear series on the original curve
\end{Th}
\begin{proof} From Lemma \ref{lem:refine-pencil}, we can assume that the series is refined and with no fixed points away from the nodes.
We will use the notation for the components and nodes as above  $C=\cup_{i\in I} C_i$  and if $C_i\cap C_{i_t}\not=\emptyset$, 
then the point of intersection comes from identifying $P_{i,i_t}$  with  $P_{i_t,i}$.

We are assuming that  $V_i$ is two dimensional and if we write  $a^0_{i,i_t}, a^1_{i,i_t}$ for the orders  of vanishing of the sections in $V_i$ at the node $P_{i,i_t}$, 
then $V_i(-\sum a^0_{i,i_t}P_{i,i_t})$ has no fixed points.
Therefore, $V_i(-\sum a^0_{i,i_t}P_{i,i_t})$ induces a map $f_i:C_i\to T_i$ where $T_i$ is a rational curve.
Write $P'_{i,i_t}$ for the image of $P_{i,i_t}$ by this map.
Construct a nodal curve  by taking $T=\cup_{i\in I} T_i$  and  identifying $P'_{i,i_t}$  with  $P'_{i_t,i}$.
As $C$ is of compact type and the dual graph of $T$ is the same as the dual graph of $C$, $T$ is a rational curve.
The inverse image of  $P'_{i,i_{t_0}}$ by $f_i$ contain  $P_{i,i_{t_0}}$ and  additional points $\bar P^1_{i,i_{t_0}},\dots , \bar P^{k(i,t_0)}_{i,i_{t_0}}$.
In fact, $k(i,t_0)$ is the degree of $L_i(-\sum_t a^0_{i,i_t}P_{i,i_t})$ minus the ramification at  $P_{i,i_{t_0}}$ of the linear series, 
that is  $d-\sum_t a^0_{i,i_t}-(a^1_{i,i_{t_0}}-a^0_{i,i_{t_0}})=d-\sum_{t\not=t_0} a^0_{i,i_t}-a^1_{i,i_{t_0}}$.
Let now $T^{i,t_0}$ be the connected component of $T-P'_{i,i_{t_0}}$ that does not contain $T_i$.
We construct $\widetilde{C}$ starting from $C$ and gluing to each of $\bar P^1_{i,i_{t_0}},\dots , \bar P^{k(i,t_0)}_{i,i_{t_0}}$ a copy of  $T^{i,t_0}$.
To construct the admissible cover from $\widetilde{C}$ to $T$, we use the functions $f_i:C_i\to T_i$ and the natural identifications of the copies of the $T^{i,t_0}$ glued to $C$ to the $T^{i,t_0}$ in $T$.

By assumption, the limit linear series is refined.
Therefore, $a^0_{i,i_t}+a^1_{i_t,i}=d, a^0_{i_t,i}+a^1_{i,i_t}=d$. Therefore, $a^1_{i,i_t}-a^0_{i,i_t}=a^1_{i_t,i}-a^0_{i_t,i}$, giving the third condition of an admissible covering.
At the new additional nodes of $C_1$, there is no ramification.

It remains to check that the map so obtained has degree $d$.
Fix a component $C_i$.
 For every other component $C_j$, there is a unique simple path in the dual graph of $C$ joining the vertex corresponding to $C_i$ to the vertex corresponding to $C_j$.
 For $j\not=i$, write $p(j)$ for the index of the vertex immediately preceding $j$ in this path.
 In particular $C_j$ and $C_{p(j)}$ intersect at the node obtained by gluing $P_{j,p(j)}\in C_j$ with $P_{p(j),j}\in C_{p(j)}$.
  This node contributes $d-\sum_ta^0_{j,t}-(a^1_{j,p(j)}-a^0_{j,p(j)})$ to the degree over $T_i$.
  The line bundle $L_i$ contributes $d-\sum_ta^0_{i,t}$ over $T_i$.
  Therefore, writing $N$ the number of components, the degree over the component $T_i$ is 
\[ d-\sum_ta^0_{i,t}+\sum _{j\not= i}(d-\sum_ta^0_{j,t}-(a^1_{j,p(j)}-a^0_{j,p(j)}))=Nd-\sum _ja^1_{j,p(j)} -\sum _{j}a^0_{p(j),j}  \]
As the linear series is refined, $a^1_{j,p(j)} +a^0_{p(j),j}=d$.
Therefore, the above sum is $d$ as needed.
\end{proof}

\section{Tropical Preliminaries}\label{sec:tropical-prelims}
We collect the tropical background and fix the notation and conventions used throughout the paper. Our conventions for metric graphs and divisor theory follow principally \cite{GK,C}, while for reduced divisors we follow \cite{L}. For harmonic morphisms, we use both the model-based and intrinsic points of view. We define harmonic morphisms through graph models, following \cite{Cap,C,MZ1}, while regarding the induced maps, slopes, and local degrees intrinsically on the underlying metric graphs, as in \cite{ABBR1,ABBR2}.

\subsection{Metric Graphs}
A graph $G$ consists of a finite vertex set $V(G)$, a finite edge set $E(G)$, and an endpoint map $\operatorname{end}_G:E(G)\to \operatorname{Sym}^2(V(G))$, where $\operatorname{Sym}^2(V(G))$ denotes unordered pairs of vertices, with repetition allowed. We abuse notation and write $e=uv$ to mean that $\operatorname{end}_G(e)=\{u,v\}$. If $e=vv$, then $e$ is called a loop edge. A length function on $G$ is a function $\ell:E(G)\to \mathbb R_{>0}$.

A metric graph $\Gamma$ is the geometric realization of a pair $(G,\ell)$, obtained by assigning to each edge $e$ of $G$ a length $\ell(e)$. In this case, $(G,\ell)$ is a model of $\Gamma$. We require metric graphs to be connected. A model is loopless if it has no loop edges.

For a point $p\in \Gamma$, the valence $\operatorname{val}_\Gamma(p)$ is the number of connected components of a sufficiently small punctured neighborhood of $p$. The points $p\in \Gamma$ with $\operatorname{val}_\Gamma(p)\neq 2$ form the canonical vertex set. If this set is empty, we choose and fix an arbitrary point of $\Gamma$ and call the resulting singleton the canonical vertex set. A finite set of points in $\Gamma$ containing the canonical vertex set is called a vertex set of $\Gamma$. A vertex set $V$ of $\Gamma$ naturally induces a model of $\Gamma$, which we denote by $(G_V(\Gamma), \ell_V(\Gamma))$. The canonical loopless vertex set is obtained from the canonical vertex set by adding the midpoint of every loop edge in the model that it induces. We denote the canonical loopless model by $(G_{-}(\Gamma),\ell_{-}(\Gamma))$ and write $V_{-}(\Gamma)$ for $V(G_{-}(\Gamma))$.

The genus of $\Gamma$ is $g(\Gamma)=b_1(\Gamma)$, where $b_1(\Gamma)$ is the first Betti number. Equivalently, for any model $(G,\ell)$ of $\Gamma$, $g(\Gamma)=|E(G)|-|V(G)|+1$.

A tropical modification of $\Gamma$ is a metric graph $\Gamma'$ containing $\Gamma$ as a closed metric subgraph such that the closure of each connected component of $\Gamma'\setminus \Gamma$ is a metric tree meeting $\Gamma$ in exactly one point. We write $\rho:\Gamma'\to \Gamma$ for the natural retraction.

\subsection{Divisors}
A divisor on $\Gamma$ is a finite formal sum $D=\sum_{x\in \Gamma}D(x)x$ with $D(x)\in \mathbb Z$. The degree of $D$ is $\deg(D)=\sum_{x\in \Gamma}D(x)$. The positive part of $D$ is $D^+=\sum_{D(x)>0}D(x)x$. If $D=D^+$, then $D$ is effective. We write $\operatorname{Div}(\Gamma)$ for the group of divisors on $\Gamma$ and $\operatorname{Div}^d(\Gamma)$ for the set of divisors of degree $d$.

A rational function on $\Gamma$ is a continuous piecewise linear function $f:\Gamma\to \mathbb R$ with integer slopes and finitely many domains of linearity. We write $\operatorname{Rat}(\Gamma)$ for the group of rational functions. The principal divisor $\operatorname{div}(f)$ is defined by declaring the coefficient of $\operatorname{div}(f)$ at a point $x\in \Gamma$ to be the sum of the outgoing slopes of $f$ at $x$. Two divisors $D$ and $D'$ are linearly equivalent, written $D\sim D'$, if their difference is principal. We write $\operatorname{Pic}(\Gamma)=\operatorname{Div}(\Gamma)/{\sim}$ and  $\operatorname{Pic}^d(\Gamma)=\operatorname{Div}^d(\Gamma)/{\sim}$.

The complete linear series of $D$ is $|D|=\{D'\geq0:D'\sim D\}$. A point $p\in \Gamma$ is a fixed point of $|D|$ if $D'(p)\geq1$ for every $D'\in |D|$.

The rank of a divisor is \[\operatorname{rk}(D)=\min\{\deg(E):E\geq0\text{ and }|D-E|=\varnothing\}-1.\]

The canonical divisor of $\Gamma$ is $K_\Gamma=\sum_{x\in\Gamma}(\operatorname{val}_\Gamma(x)-2)x$. A divisor $D$ is nonspecial if $\operatorname{rk}(K_\Gamma-D)=-1$.

Let $\rho:\Gamma'\to\Gamma$ be a tropical modification. The retraction induces a degree-preserving homomorphism $\rho_*:\operatorname{Div}(\Gamma')\to\operatorname{Div}(\Gamma)$ which sends principal divisors to principal divisors and induces an isomorphism $\rho_*:\operatorname{Pic}(\Gamma')\xrightarrow{\sim}\operatorname{Pic}(\Gamma)$. Moreover, $\operatorname{rk}_{\Gamma'}(D)=\operatorname{rk}_{\Gamma}(\rho_*(D))$ for every $D\in\operatorname{Div}(\Gamma')$.

\subsection{Reduced Divisors}
Let $D\in\operatorname{Div}(\Gamma)$ and $q\in\Gamma$.  A boundary point $x\in\partial A$ of a closed connected subset $A\subseteq\Gamma$ is $D$-saturated with respect to $A$ if $D(x)\geq\operatorname{outdeg}_A(x)$, where $\operatorname{outdeg}_A(x)$ is the number of tangent directions at $x$ leaving $A$. A divisor is $q$-reduced if it is effective away from $q$ and every nonempty closed connected subset of $\Gamma\setminus\{q\}$ has a boundary point that is not $D$-saturated. We write $D_q$ for the unique $q$-reduced divisor linearly equivalent to $D$.

A subset $S\subseteq\Gamma$ is rank-determining if, for every divisor $D$ and every $r\geq0$, one has $\operatorname{rk}(D)\geq r$ if and only if $|D-E|\neq\varnothing$ for every effective divisor $E$ of degree $r$ supported on $S$.

\begin{Lemma}\label{lem:rank-determining-criterion}
One has $\operatorname{rk}(D)\geq1$ if and only if $D_v(v)\geq1$ for every $v\in V_{-}(\Gamma)$.
\end{Lemma}

\begin{proof}
    By \cite[Theorem~1.5]{L}, the canonical loopless vertex set $V_{-}(\Gamma)$ is rank-determining. The lemma then follows from the fact that $|D-v|\neq\varnothing$ if and only if $D_v(v)\geq1$.
\end{proof}

\begin{Lemma}\label{lem:reduced-maximum-locus}
Let $D_q$ be $q$-reduced, and suppose that $f\in\operatorname{Rat}(\Gamma)$ satisfies $D_q+\operatorname{div}(f)\geq0$. Then $f$ attains its maximum at $q$.
\end{Lemma}

\begin{proof}
    This is \cite[Lemma~7]{Amini}.
\end{proof}

\begin{Lemma}\label{lem:agreement-on-components-general}
Let $s\in\Gamma$, let $U$ be a connected component of $\Gamma\setminus\{s\}$, and let $q,q'\in\Gamma\setminus U$. Then, $D_q|_U=D_{q'}|_U$.
\end{Lemma}

\begin{proof}
Choose $f\in\operatorname{Rat}(\Gamma)$ such that $\operatorname{div}(f)=D_{q'}-D_q$. Define $g\in\operatorname{Rat}(\Gamma)$ by setting $g=f$ on $U$ and $g=f(s)$ on $\Gamma\setminus U$.

Suppose that $g>f(s)$ somewhere on $U$, and let $M$ be a connected component of the maximum locus of $g$ at a value greater than $f(s)$. Then $M\subseteq U$ and $q\notin M$. Since $D_q$ is $q$-reduced, some $u\in\partial M$ is not $D_q$-saturated. Every outgoing slope of $g$ from $M$ at $u$ is at most $-1$, so $D_q(u)+\operatorname{div}(g)(u)<0$. Since $f=g$ near $M$, this gives $D_{q'}(u)=D_q(u)+\operatorname{div}(f)(u)<0$, contrary to the effectiveness of $D_{q'}$ away from $q'$. Hence $f\leq f(s)$ on $U$.

Applying the same argument to $-f$, with $q$ and $q'$ interchanged, gives $f\geq f(s)$ on $U$. Thus $f$ is constant on $U$, and the result follows.
\end{proof}

\subsection{Harmonic Morphisms}
Let $\Gamma_1$ and $\Gamma_2$ be metric graphs with loopless models $(G_1,\ell_1)$ and $(G_2,\ell_2)$. A morphism of loopless models $\varphi:(G_1,\ell_1)\to(G_2,\ell_2)$ is a continuous map $\varphi:\Gamma_1\to\Gamma_2$ such that $\varphi(V(G_1))\subseteq V(G_2)$ and every edge $e=uv$ of $G_1$ satisfies exactly one of the following:
\begin{enumerate}[label=(\roman*)]
\item The edge $e$ is contracted to a vertex of $G_2$, in which case $\varphi(e)=\varphi(u)=\varphi(v)$. We then set $\mu_\varphi(e)=0$.

\item The edge $e$ maps linearly onto an edge $e'$ of $G_2$, with $e'=\varphi(u)\varphi(v)$, and $\mu_\varphi(e)=\frac{\ell_2(e')}{\ell_1(e)}\in\mathbb Z_{>0}$. The integer $\mu_\varphi(e)$ is called the expansion factor of $\varphi$ along $e$.
\end{enumerate}

A continuous map $\varphi:\Gamma_1\to\Gamma_2$ is a morphism of metric graphs if it is induced by a morphism of loopless models. A morphism of loopless models is harmonic at a vertex $v\in V(G_1)$ if the sum $\sum_{\substack{e\in E_v(G_1)\\ \varphi(e)=e'}}\mu_\varphi(e)$ is independent of $e'\in E_{\varphi(v)}(G_2)$. A morphism of metric graphs is harmonic if it is induced by a harmonic morphism of loopless models. 

If the target has at least one edge, this sum is denoted $m_\varphi(v)$ and called the local degree of $\varphi$ at $v$; the local degree at any point of the source is defined after subdivision. If $\Gamma_2=\{t\}$, then, following \cite[Remark~2.7]{ABBR1}, the data of $\varphi$ additionally includes an effective divisor $F_\varphi\in\operatorname{Div}(\Gamma_1)$, and we set $m_\varphi(x)=F_\varphi(x)$ for every $x\in\Gamma_1$. The degree of $\varphi$ is $\deg(\varphi)=\sum_{x \in \varphi^{-1}(x')}m_\varphi(x)$, where $x'$ is any point in $\Gamma_2$.

A harmonic morphism is finite if $m_\varphi(x) \geq 1$ for every $x \in \Gamma_1$. A harmonic morphism is non-degenerate if there exists a loopless vertex set $V\subseteq\Gamma_1$ such that $m_\varphi(v)\geq1$ for every $v\in V$. This is equivalent to the model-based definition of non-degeneracy in \cite{MZ1}.

The Riemann--Hurwitz divisor of a harmonic morphism $\varphi:\Gamma_1\to\Gamma_2$ is $R_\varphi= K_{\Gamma_1} - \varphi^*(K_{\Gamma_2})$. The morphism $\varphi$ is a tropical cover if $R_\varphi$ is effective.

A finite tropical cover is called a tropical admissible cover. Following \cite[Definition~2.5]{MaZ}, a well-contracted cover is a non-degenerate tropical cover whose target is a path. Every non-degenerate harmonic morphism to a path is a well-contracted cover, since the Riemann--Hurwitz condition is automatic for a path target.

For a harmonic morphism $\varphi: \Gamma_1 \to \Gamma_2$, define the pullback of a rational function $f \in \operatorname{Rat}(\Gamma_2)$ by $\varphi^*(f)=f \circ \varphi \in \operatorname{Rat}(\Gamma_1)$ and define the pullback of a divisor $D \in \operatorname{Div}(\Gamma_2)$ by $\varphi^*(D)=\sum_{y\in\Gamma_2}D(y)\sum_{x\in\varphi^{-1}(y)}m_\varphi(x)x$. By \cite[Proposition~2.7]{C}, $\operatorname{div}(\varphi^*(f))=\varphi^*(\operatorname{div}(f))$, so pulling back along a harmonic morphism preserves linear equivalence.

Following \cite[Definition~8]{MZ1}, a tropical modification of a harmonic morphism $\varphi:\Gamma_1\to\Gamma_2$ is a harmonic morphism $\widetilde{\varphi}:\widetilde{\Gamma}_1\to\widetilde{\Gamma}_2$ such that each $\widetilde{\Gamma}_i$ is a tropical modification of $\Gamma_i$ and the diagram
$$
\begin{CD}
\widetilde{\Gamma}_1 @>{\widetilde{\varphi}}>> \widetilde{\Gamma}_2 \\
@V{\rho_1}VV @VV{\rho_2}V \\
\Gamma_1 @>{\varphi}>> \Gamma_2
\end{CD}
$$ 
commutes. We additionally require $(\rho_1)_*(\widetilde{\varphi}^*(t)) \sim \varphi^*(\rho_2(t))$ for every $t\in\widetilde{\Gamma}_2$. When $\Gamma_2$ is not a point, this condition is automatic, while for a point target it is an additional requirement.

We make repeated use of the following extension of \cite[Proposition~2.4]{MZ1}.

\begin{Lemma}\label{lem:finite-tropical-modification}
Every non-degenerate harmonic morphism, respectively well-contracted cover, admits a finite harmonic morphism, respectively tropical admissible cover, as a tropical modification.
\end{Lemma}

\begin{proof}
Let $\varphi:\Gamma_1\to\Gamma_2$ be a non-degenerate harmonic morphism. By definition, there exists a loopless model $(G,\ell)$ of $\Gamma_1$ such that $m_\varphi(v)\geq1$ for every $v\in V(G)$ and $\varphi$ is a morphism of loopless models from $(G,\ell)$ to some model of $\Gamma_2$. If $\Gamma_2$ is a point, we further take $V(G)$ to contain $\operatorname{Supp}(F_\varphi)$.

First, if $\deg(\varphi)=1$, then $\varphi$ is already finite. Indeed, suppose some edge $uv$ of $G$ is contracted by $\varphi$. Then $\deg(\varphi)=\sum_{x\in\varphi^{-1}(\varphi(u))}m_\varphi(x)\geq m_\varphi(u)+m_\varphi(v)\geq2$, a contradiction.

Now suppose $\deg(\varphi)\geq2$. We apply the construction of \cite[Proposition~2.4]{MZ1}. For every contracted edge $uv$ of $G$, attach in the target an edge of length $\frac{\ell(uv)}{2}$ at $\varphi(uv)$. In the source, attach $m_\varphi(u)-1$ and $m_\varphi(v)-1$ edges of length $\frac{\ell(uv)}{2}$ at $u$ and $v$, respectively, and $m_\varphi(x)$ such edges at every $x\in\varphi^{-1}(\varphi(uv))\setminus\{u,v\}$. Map each added edge, together with each half of $uv$, identically onto the added target edge. This gives a finite harmonic morphism $\widetilde{\varphi}:\widetilde{\Gamma}_1\to\widetilde{\Gamma}_2$, which is a tropical modification of $\varphi$. In the point-target case, the additional condition follows since the fiber over the original target point pushes forward to $F_\varphi$, while all fibers of $\widetilde{\varphi}$ are linearly equivalent.

Suppose now that $\varphi$ is well-contracted. Let $x\in V(G)$ and let $a$ and $c$ be the numbers of non-contracted and contracted edges incident to $x$, respectively. If $k=\operatorname{val}_{\Gamma_2}(\varphi(x))$ and $s$ contracted edges of $G$ map to $\varphi(x)$, then $\operatorname{val}_{\widetilde{\Gamma}_2}(\widetilde{\varphi}(x))=k+s$. At $x$, the construction adds $sm_\varphi(x)-c$ edges, so $\operatorname{val}_{\widetilde{\Gamma}_1}(x)=a+sm_\varphi(x)$. Thus $R_{\widetilde{\varphi}}(x) =a-2-m_\varphi(x)(k-2)$. Since $\Gamma_2$ is a path, $k\leq2$, while harmonicity gives $a\geq k$ and non-degeneracy gives $m_\varphi(x)\geq1$. Hence this coefficient is nonnegative.

Every point of $\widetilde{\Gamma}_1\setminus\Gamma_1$ has coefficient $0$, while every point of $\Gamma_1\setminus V(G)$ has coefficient $0$ except the midpoint of each formerly contracted edge, where the coefficient is $2$. Thus $R_{\widetilde{\varphi}}$ is effective, so $\widetilde{\varphi}$ is a tropical admissible cover.
\end{proof}

\begin{Ex}
We give an example of a non-degenerate harmonic morphism and the finite tropical modification produced by the construction of \cite[Proposition~2.4]{MZ1}. In the following figure, the horizontal maps send colored source edges to target edges of the same color and contract black source edges. The vertical arrows are the retractions. Edge labels indicate expansion factors, and unlabeled non-contracted edges have expansion factor $1$. Edge lengths are chosen compatibly with the indicated morphisms.

\input{finitization-figure}
\end{Ex}

Lastly, let $\varphi:\Gamma_1\to T$ be a non-degenerate harmonic morphism to a metric tree. Since any two points of $T$ are linearly equivalent and pullback preserves linear equivalence, all fibers of $\varphi$ are linearly equivalent. Moreover, every fiber of $\varphi$ has rank at least $1$. Indeed, by Lemma~\ref{lem:finite-tropical-modification}, $\varphi$ admits a finite tropical modification $\widetilde{\varphi}:\widetilde{\Gamma}_1\to\widetilde T$. Every fiber of $\widetilde{\varphi}$ has rank at least $1$, since for every $x\in\widetilde{\Gamma}_1$ the divisor $\widetilde{\varphi}^*(\widetilde{\varphi}(x))-x$ is effective. By the definition of a tropical modification of a harmonic morphism, the retraction of each fiber of $\widetilde{\varphi}$ is linearly equivalent to the corresponding fiber of $\varphi$. Since retraction preserves rank, every fiber of $\varphi$ has rank at least $1$.
\subsection{Gonality}

A metric graph $\Gamma$ is $d$-gonal if there is a degree-$d$ non-degenerate harmonic morphism from a tropical modification of $\Gamma$ to a metric tree.

A metric graph $\Gamma$ is divisorially $d$-gonal if it has a degree-$d$ divisor of rank at least $1$.

Every $d$-gonal metric graph is divisorially $d$-gonal, since the retraction of any fiber of the exhibiting harmonic morphism has degree $d$ and rank at least $1$.

 Chan showed in \cite[Theorem~1.3]{C} that a metric graph is $2$-gonal if and only if it is divisorially $2$-gonal. We call a metric graph satisfying these equivalent properties hyperelliptic.

\section{Tree-Induced Divisors}\label{sec:tree-induced-divisors}

Let $D$ be a divisor on a metric graph $\Gamma$. We say that $D$ is tree-induced if there exist a tropical modification $\rho:\Gamma'\to\Gamma$, a non-degenerate harmonic morphism $\varphi:\Gamma'\to T$ to a metric tree, and $t\in T$ such that $\rho_*(\varphi^*(t))\sim D$. We say that $D$ is path-induced if $T$ may be chosen to be a metric path.

Let $X$ be a smooth proper curve over a non-Archimedean field $K$, let $\Gamma\subseteq X^{\operatorname{an}}$ be a skeleton, and let $\tau: X^{\operatorname{an}}\to\Gamma$ be the retraction. For a divisor $B=\sum_i n_iP_i$ on $X$, we call $\tau_*(B)=\sum_i n_i\tau(P_i)$ the tropicalization of $B$ on $\Gamma$.

The following proposition characterizes tree-induced divisors algebraically.

\begin{Prop}\label{prop:tree-induced-lift}
A divisor $D$ on $\Gamma$ is tree-induced if and only if there exist an algebraically closed complete non-Archimedean field $K$, a smooth proper connected curve $X$ over $K$ having $\Gamma$ as a skeleton, and a divisor $B$ on $X$ such that $\operatorname{rk}_X(B)\geq1$ and the tropicalization of $B$ is linearly equivalent to $D$.
\end{Prop}

\begin{proof}
Set $d=\deg(D)$. By Lemma~\ref{lem:finite-tropical-modification} and \cite[Theorem~8.6]{AG}, the tree-induced divisors of degree $d$ are precisely, up to linear equivalence, the underlying divisors of effective combinatorial $g^1_d$'s. By \cite[Theorem~8.8]{AG}, every such combinatorial $g^1_d$ is induced by an algebraic divisor of rank at least $1$. Conversely, \cite[Theorem~9.1]{AG} associates an effective combinatorial $g^1_d$ to every algebraic divisor of rank at least $1$. The result follows.
\end{proof}

For path-induced divisors, the forward implication can be made genus-preserving.

\begin{Prop}\label{prop:path-induced-lift}
Let $D$ be a path-induced divisor on $\Gamma$. Then there exist an algebraically closed complete non-Archimedean field $K$, a smooth proper connected curve $X$ over $K$ having $\Gamma$ as a skeleton, and a divisor $B$ on $X$ such that $\operatorname{rk}_X(B)\geq1$, the tropicalization of $B$ is linearly equivalent to $D$, and $g(X)=g(\Gamma)$.
\end{Prop}

\begin{proof}
Choose a well-contracted cover $\varphi:\Gamma'\to P$ exhibiting the path-inducibility of $D$. By Lemma~\ref{lem:finite-tropical-modification}, it admits a tropical admissible cover $\widetilde\varphi:\widetilde\Gamma\to\widetilde P$ as a tropical modification. Choose an algebraically closed complete non-Archimedean field $K$ of residue characteristic $0$ whose value group contains all edge lengths of $\widetilde\Gamma$ and $\widetilde P$. Since $P$ is a path and the finitization adds only trivial ramification profiles, each local Hurwitz problem has at most two nontrivial profiles. Double Hurwitz numbers are nonzero, so \cite[Corollary~3.8]{ABBR2} gives a lift in which all source and target vertex curves are rational. Since tropical modification preserves genus, the source curve $X$ has genus $g(\Gamma)$, while the lifted target has genus $0$. A fiber of the lifted morphism then gives the required divisor $B$.
\end{proof}

We also record that tree-inducibility and path-inducibility are preserved under addition of effective divisors.

\begin{Lemma}\label{lem:effective-extension}
Let $A$ be a tree-induced, respectively path-induced, divisor on $\Gamma$, and let $E$ be an effective divisor on $\Gamma$. Then $A+E$ is tree-induced, respectively path-induced.
\end{Lemma}

\begin{proof}
Let $\varphi:\Gamma'\to T$ exhibit $A$ as tree-induced, with $T$ a path in the path-induced case, and let $\rho:\Gamma'\to\Gamma$ be the retraction. If $T$ is a point, replace $F_\varphi$ by $F_\varphi+E$. Otherwise, for each $p\in\operatorname{supp}(E)$, attach $E(p)$ copies of $T$ to $\Gamma'$ at $p$, identifying $p$ with the point corresponding to $\varphi(p)$, and map each copy isometrically to $T$. The resulting morphism $\widetilde\varphi:\widetilde\Gamma\to T$ is non-degenerate and harmonic. If $\widetilde\rho:\widetilde\Gamma\to\Gamma$ is the retraction, then $\widetilde\rho_*(\widetilde\varphi^*(t))=\rho_*(\varphi^*(t))+E\sim A+E$ for every $t\in T$. Since the target is unchanged, the construction preserves path-inducibility.
\end{proof}

The following elementary bounds show that tree-inducibility already holds in several favorable situations.

\begin{Prop}\label{prop:gonality-tree-induced-bounds}
Let $\Gamma$ be a $k$-gonal metric graph of genus $g$, and let $D$ be a positive-rank divisor on $\Gamma$. Then $D$ is tree-induced if any of the following holds:
\begin{enumerate}[label=(\roman*)]
\item $D$ is nonspecial;
\item $\operatorname{rk}(D)\geq k -1$;
\item $\deg(D)\geq g+k-2$.
\end{enumerate}
\end{Prop}

\begin{proof}
Suppose first that $D$ is nonspecial. Choose an algebraically closed complete non-Archimedean field $K$ with value group $\mathbb R$ and residue field $\kappa$ of characteristic $0$, and assign $\mathbb P^1_\kappa$ to every vertex of $G_{-}(\Gamma)$. By \cite[Theorem~3.24]{ABBR1}, the resulting metrized complex is the skeleton of a smooth proper connected curve $X/K$ of genus $g$. Lift an effective divisor equivalent to $D$ pointwise to a divisor $B$ on $X$. Tropical Riemann--Roch gives $\operatorname{rk}(D)=\deg(D)-g$. The specialization lemma \cite[Theorem~5.11]{AB}, together with the fact that metrized-complex rank is at most the rank of the underlying metric-graph divisor, gives $\operatorname{rk}_X(B)\leq\operatorname{rk}(D)$, while algebraic Riemann--Roch gives $\operatorname{rk}_X(B)\geq\deg(D)-g$. Hence $\operatorname{rk}_X(B)=\operatorname{rk}(D)\geq1$. Proposition~\ref{prop:tree-induced-lift} therefore shows that $D$ is tree-induced.

Now suppose that $D$ is special. Since $\Gamma$ is $k$-gonal, there is a tree-induced divisor $A$ of degree $k$ and rank $a\geq1$. If $|D-A|\neq\varnothing$, then $D\sim A+E$ for some effective divisor $E$, and Lemma~\ref{lem:effective-extension} applies. Otherwise, $\operatorname{rk}(D-A)=-1$, so Riemann--Roch and superadditivity of rank give $g+ k -\deg(D)-2=\operatorname{rk}(K_\Gamma-D+A)\geq\operatorname{rk}(K_\Gamma-D)+a=\operatorname{rk}(D)-\deg(D)-1+g+a$. Hence $\operatorname{rk}(D)\leq k -a-1\leq k -2$, while $\operatorname{rk}(K_\Gamma-D+A)\geq a$ gives $\deg(D)\leq g+k-a-2\leq g+k-3$. Thus either (ii) or (iii) forces $|D-A|\neq\varnothing$, and the result follows.
\end{proof}

\begin{Cor}\label{cor:hyperelliptic-tree-induced}
On hyperelliptic metric graphs, every positive-rank divisor is tree-induced.
\end{Cor}
\section{Metric Graphs of Pseudocompact Type}\label{sec:pseudocompact-type}

The simplification $G^{\mathrm{simp}}$ of a loopless graph $G$ is obtained by replacing each nonempty family of parallel edges by a single edge. A multitree is a connected loopless graph whose simplification is a tree, and a multipath is a connected loopless graph whose simplification is a path. A metric graph $\Gamma$ is of pseudocompact type if $G_{-}(\Gamma)$ is a multitree. The terminology is borrowed from Osserman's notion of a nodal curve of pseudocompact type, whose dual graph satisfies the same condition \cite{Oss}. If $G_{-}(\Gamma)$ is a multipath, then we call $\Gamma$ a metric banana path.

Throughout this section, we fix a metric graph of pseudocompact type $\Gamma$ and a positive-rank divisor $D\in\operatorname{Div}(\Gamma)$. Additionally, let $T=G_{-}(\Gamma)^{\mathrm{simp}}$. We identify $V(T)$ with $V_{-}(\Gamma)$. For each edge $uv\in E(T)$, let $B_{uv}\subseteq\Gamma$ be the closed subgraph consisting of the edges of $G_{-}(\Gamma)$ with endpoints $u$ and $v$. 

\begin{Lemma}\label{lem:transition-functions}
Let $u,v\in V_{-}(\Gamma)$ be adjacent, and let
$f_{uv}\in\operatorname{Rat}(\Gamma)$ satisfy
$\operatorname{div}(f_{uv})=D_u-D_v$. Then:
\begin{enumerate}[label=(\roman*)]
\item Every slope of $f_{uv}$ along an edge not contained in $B_{uv}$ is zero.
\item The function $f_{uv}$ attains its minimum at $u$ and its maximum at $v$.
\end{enumerate}
\end{Lemma}

\begin{proof}
    Lemma~\ref{lem:agreement-on-components-general}, applied at $u$ and $v$, gives (i). Applying Lemma~\ref{lem:reduced-maximum-locus} to $D_v+\operatorname{div}(f_{uv})=D_u$ and $D_u+\operatorname{div}(-f_{uv})=D_v$ gives (ii).
\end{proof}

We now prove Theorem~\ref{thm:pseudocompact-tree-induced}.

\begin{proof}
Choose a root $r\in V_{-}(\Gamma)$ and orient $T$ away from $r$. For each oriented edge $u\to v$ of $T$, choose $f_{uv}\in\operatorname{Rat}(\Gamma)$ as in Lemma~\ref{lem:transition-functions} and define $f=\sum_{u\to v}f_{uv}$, where the sum is over the edges of $T$ with the orientation away from $r$.

Fix an oriented edge $u\to v$. By Lemma~\ref{lem:transition-functions}(i), every summand other than $f_{uv}$ has zero slope on $B_{uv}$. Hence $f|_{B_{uv}}$ and $f_{uv}|_{B_{uv}}$ differ by a constant, and Lemma~\ref{lem:transition-functions}(ii) gives $f(u)\leq f(p)\leq f(v)$ for every $p\in B_{uv}$.

Suppose first that $f$ is constant. Then every $f_{uv}$ is constant on $B_{uv}$ and has zero slope outside $B_{uv}$, hence is constant on $\Gamma$. Thus $D_u=D_v$ for every edge $uv$ of $T$, so $D_v=D_r$ for every $v\in V_{-}(\Gamma)$. By Lemma~\ref{lem:rank-determining-criterion}, $D_r(v)=D_v(v)\geq1$ for every $v\in V_{-}(\Gamma)$. Thus $D_r$ is effective. The constant harmonic morphism $\varphi:\Gamma\to\{t\}$ with $F_\varphi=D_r$ is non-degenerate and satisfies $\varphi^*(t)=D_r\sim D$, so $D$ is tree-induced.

Assume now that $f$ is nonconstant. Contract every edge $uv$ of $T$ for which $f(u)=f(v)$, and let $\pi:T\to T'$ be the quotient map. If $uv$ is contracted, then the preceding inequalities show that $f$ is constant on $B_{uv}$. Since $f|_{B_{uv}}$ and $f_{uv}|_{B_{uv}}$ differ by a constant, $f_{uv}$ is constant on $B_{uv}$ and hence on $\Gamma$, so $D_u=D_v$. Therefore, for $x\in V(T')$, the definitions $\bar f(x)=f(u)\quad\text{and}\quad D_x=D_u$ for $u\in\pi^{-1}(x)$ are independent of the choice of $u$. Moreover, $D_x$ is effective, since $D_u$ is effective away from $u$ and Lemma~\ref{lem:rank-determining-criterion} gives $D_u(u)\geq1$.

Orient $T'$ away from $\pi(r)$. Whenever $x\to y$ is an oriented edge of $T'$, write $u\to v$ for its unique noncontracted preimage edge in $T$, so that $\pi(u)=x$ and $\pi(v)=y$. Give $xy$ length $\ell(xy)=\bar f(y)-\bar f(x)>0$, and let $\Omega=(T',\ell)$. Define $\varphi_0:\Gamma\to\Omega$ by $\varphi_0(u)=\pi(u)$ for $u\in V(T)$. On each $B_{uv}$ with $f(u)\neq f(v)$, map $p$ to the point of $xy$ at distance $f(p)-f(u)$ from $x$. Contract every $B_{uv}$ with $f(u) = f(v)$. After suitable subdivisions, this is a morphism of metric graphs.

For $p\in\Gamma$, define
\[\operatorname{Adj}(p)= 
\begin{cases}
\{x,y\}&\text{ if } \varphi_0(p)\text{ lies in the interior of an edge }xy\text{ of }T';\\
\{y\in V(T'):xy\in E(T')\}&\text{ if } \varphi_0(p)=x\in V(T').
\end{cases} \]
 For $q\in\operatorname{Adj}(p)$, let $\Omega_{p,q}$ be the closure of the component of $\Omega\setminus\{\varphi_0(p)\}$ containing $q$. Attach at $p$ exactly $D_q(p)$ copies of $\Omega_{p,q}$ and map each copy isometrically to $\Omega_{p,q}$, omitting the attachment when $D_q(p)=0$. Let $\rho:\Gamma'\to\Gamma$ be the resulting tropical modification and $\varphi:\Gamma'\to\Omega$ the resulting morphism.

Harmonicity is automatic on the added trees, so fix $p\in\Gamma$. For $q\in\operatorname{Adj}(p)$, let $a_q(p)$ be the sum of the expansion factors along the tangent directions at $p$ mapped by $\varphi_0$ toward $q$. The corresponding directional degree of $\varphi$ is $a_q(p)+D_q(p)$. If $\varphi_0(p)$ lies in the interior of an edge $xy$ of $T'$, let $u\to v$ be its noncontracted preimage edge in $T$. Since $\operatorname{div}(f_{uv})=D_u-D_v=D_x-D_y$, while the sums of the outgoing slopes of $f_{uv}$ at $p$ in the directions mapped toward $y$ and $x$ are $a_y(p)$ and $-a_x(p)$, respectively, we obtain $D_x(p)-D_y(p)=a_y(p)-a_x(p)$. Thus $a_x(p)+D_x(p)=a_y(p)+D_y(p)$. If instead $\varphi_0(p)=x\in V(T')$, let $q$ be adjacent to $x$, and order the unique preimage edge $uv$ of $xq$ so that $\pi(u)=x$ and $\pi(v)=q$. Then $\operatorname{div}(f_{uv})=D_u-D_v=D_x-D_q$. By Lemma~\ref{lem:transition-functions} and the definition of $\varphi_0$, the nonzero outgoing slopes of $f_{uv}$ at $p$ are precisely those mapped toward $q$, and their sum is $a_q(p)$. Hence $a_q(p)=D_x(p)-D_q(p)$, so $a_q(p)+D_q(p)=D_x(p)$, independently of $q$. Therefore $\varphi$ is harmonic.

The morphism $\varphi$ is non-degenerate. Indeed, for every $v\in V_{-}(\Gamma)$, $m_\varphi(v)=D_{\pi(v)}(v)=D_v(v)\geq1$ by Lemma~\ref{lem:rank-determining-criterion}, while every vertex of $G_{-}(\Gamma')$ not belonging to $V_{-}(\Gamma)$ either lies on an attached tree or is an attachment point at which such a tree contributes positive local degree.

Finally, fix $z\in V(T')$ and compute the coefficient of $\rho_*(\varphi^*(z))$ at each $p\in\Gamma$. If $\varphi_0(p)=z$, this coefficient is $m_\varphi(p)=D_z(p)$. Otherwise, let $q\in\operatorname{Adj}(p)$ be the unique vertex in the direction of $z$. Precisely the $D_q(p)$ copies of $\Omega_{p,q}$ contain a point above $z$, and we claim that $D_q(p)=D_z(p)$. This is immediate if $q=z$. Otherwise, let $uv$ be the unique noncontracted edge of $T$ mapping to the first edge in the direction from $\varphi_0(p)$ toward $z$, ordered so that $v\in\pi^{-1}(q)$, and choose $w\in\pi^{-1}(z)$. The component $U$ of $\Gamma\setminus\{v\}$ containing $p$ does not contain $w$, so Lemma~\ref{lem:agreement-on-components-general} gives $D_v|_U=D_w|_U$. Since $D_q=D_v$ and $D_z=D_w$, it follows that $D_q(p)=D_z(p)$. Thus the coefficient of $\rho_*(\varphi^*(z))$ at $p$ is $D_z(p)$ in every case. Hence $\rho_*(\varphi^*(z))=D_z\sim D$, so $D$ is tree-induced.
\end{proof}

\begin{Cor}\label{cor:pseudocompact-gonality-equivalence}
A metric graph of pseudocompact type is $d$-gonal if and only if it is divisorially $d$-gonal.
\end{Cor}

\begin{Cor}\label{cor:pseudocompact-lifting}
Let $\Gamma$ be a metric graph of pseudocompact type, and let $D\in\operatorname{Div}(\Gamma)$ have positive rank. Then there exist an algebraically closed complete non-Archimedean field $K$, a smooth proper connected curve $X$ over $K$ having $\Gamma$ as a skeleton, and a divisor $B$ on $X$ such that $\operatorname{rk}_X(B)\geq1$ and the tropicalization of $B$ is linearly equivalent to $D$. If $\Gamma$ is a metric banana path, then $X$ may be chosen so that $g(X)=g(\Gamma)$.
\end{Cor}

\begin{proof}
The first statement follows from Theorem~\ref{thm:pseudocompact-tree-induced} and Proposition~\ref{prop:tree-induced-lift}. If $\Gamma$ is a metric banana path, then $T$ is a path, and every quotient of $T$ obtained by contracting edges is again a path. Thus the proof of Theorem~\ref{thm:pseudocompact-tree-induced} exhibits every positive-rank divisor as path-induced. Proposition~\ref{prop:path-induced-lift} then gives the second statement.
\end{proof}


\section{Gonality of metric banana paths}\label{sec:MBP-BN-theory}

We now study the Brill--Noether theory of metric banana paths. 
Related work on graphs of a similar type appears in \cite{PS} and the forthcoming paper \cite{Lo}.

Let $\Gamma$ be a metric banana path, and write $B_1,\dots,B_N$ for its bananas in order. Write $V(G_{-}(\Gamma)^{\mathrm{simp}})=\{v_0,\dots,v_N\}$, ordered along the path, so that $B_i$ has canonical vertices $v_{i-1}$ and $v_i$. Two metric banana paths with the same $N$ and the same genera $g(B_i)$ will be said to have the same combinatorial type.

For a fixed combinatorial type $G$, we identify the space of metrics with $\mathbb R_{>0}^{E(G)}$. All loci and codimensions in this section are taken in this space.

We write $\rho(g,d,r)=g-(r+1)(g-d+r)$.

\begin{Lemma}\label{lem:rank-one-gluing}
Let $\Gamma_1,\Gamma_2$ be metric graphs, and let $\Gamma$ be obtained by identifying $v\in\Gamma_1$ with $v'\in\Gamma_2$, writing $q$ for the resulting point.
\begin{enumerate}[label=(\roman*)]
\item If $D$ and $D'$ have rank at least $1$ on $\Gamma_1$ and $\Gamma_2$, respectively, then $D+D'-q$ has rank at least $1$ on $\Gamma$.
\item If $\Gamma_1$ is not Brill--Noether general for $r=1$, then $\Gamma$ is not Brill--Noether general for $r=1$.
\end{enumerate}
\end{Lemma}

\begin{proof}
For (i), choose effective divisors $E,E'$ such that $D\sim v+E$ and $D'\sim v'+E'$. If $x\in\Gamma_1$, choose an effective divisor $\widetilde D\sim D$ containing $x$; then $\widetilde D+E'\sim D+D'-q$ contains $x$. The same argument applies if $x\in\Gamma_2$, so $D+D'-q$ has rank at least $1$.

For (ii), write $g_i=g(\Gamma_i)$ and suppose $\dim W^1_d(\Gamma_1)>\rho(g_1,d,1)$. Set $e=g_2+1$. By Riemann--Roch, $W^1_e(\Gamma_2)=\operatorname{Pic}^e(\Gamma_2)$, which has dimension $g_2=\rho(g_2,e,1)$. The cut-point decomposition $\operatorname{Pic}^0(\Gamma)\simeq\operatorname{Pic}^0(\Gamma_1)\times\operatorname{Pic}^0(\Gamma_2)$ and (i) give an embedding
\[
W^1_d(\Gamma_1)\times\operatorname{Pic}^e(\Gamma_2)\longrightarrow W^1_{d+e-1}(\Gamma).
\]
Hence
\[
\dim W^1_{d+e-1}(\Gamma)\geq\dim W^1_d(\Gamma_1)+g_2>\rho(g_1+g_2,d+e-1,1),
\]
so $\Gamma$ is not Brill--Noether general for $r=1$.
\end{proof}

\begin{Cor}\label{NotBNgen}
If a metric banana path contains a banana of genus at least $3$, then it is not Brill--Noether general.
\end{Cor}

\begin{proof}
Let $B$ be such a banana, with canonical vertices $v,w$. The involution interchanging $v$ and $w$ and folding each edge gives a degree-$2$ harmonic morphism from $B$ to a tree, so $v+w$ has rank at least $1$. Thus $W^1_2(B)\neq\varnothing$, while $\rho(g(B),2,1)=2-g(B)<0$. Hence $B$ is not Brill--Noether general, and the result follows from Lemma~\ref{lem:rank-one-gluing}(ii).
\end{proof}

We now restrict to metric banana paths with $N$ bananas, all of genus $2$. Thus $g(\Gamma)=2N$.

\begin{Def}\label{def:generic-banana}
For each $i$, write $\ell_{i,0},\ell_{i,1},\ell_{i,2}$ for the edge lengths of $B_i$. We say that $\Gamma$ is generic if, whenever $\ell_{i,j}/\ell_{i,k}=p/q$ for $j\neq k$ and relatively prime positive integers $p,q$, one has $p+q>N-1$.
\end{Def}

If $C\subset B_i$ consists of two edges whose length ratio is $p/q$ in lowest terms, then $[v_{i-1}-v_i]$ has order $p+q$ in $\operatorname{Pic}^0(C)$. Hence, if $\Gamma$ is generic, then $n v_{i-1}\not\sim_C n v_i$ for every genus-$1$ subgraph $C\subset B_i$ and every $1\leq n\leq N-1$.

\begin{Lemma}\label{LSBanana}
Let $B$ be a genus-$2$ banana with canonical vertices $v,w$.
\begin{enumerate}[label=(\roman*)]
\item If $B$ is a banana in a generic metric banana path with $N$ genus-$2$ bananas and $1\leq a\leq N$, then $av\not\sim aw$ and $av\not\sim(a-1)w+x$ for every $x\in B\setminus\{v,w\}$.
\item For every $a\geq2$, there is an effective divisor $E$ of degree $2$ such that $av\sim(a-2)w+E$.
\item For every $a\geq3$, the condition $av\sim aw$ occurs on a codimension-$2$ locus of edge lengths, while $av\sim(a-1)w+x$ for some $x\in B$ occurs on a nonempty relatively open subset of a codimension-$1$ locus.
\end{enumerate}
\end{Lemma}

\begin{proof}
For (i), suppose that either $av\sim aw$ or $av\sim(a-1)w+x$. In the first case choose any two edges, and in the second choose the two edges not containing $x$. On these two edges a function giving the equivalence is linear with positive integral slopes $m_i,m_j$ satisfying $m_i+m_j\leq a-1$. Continuity gives $m_i\ell_i=m_j\ell_j$, so the order of $[v-w]$ on the corresponding genus-$1$ subgraph is at most $m_i+m_j\leq a-1\leq N-1$, contradicting genericity.

For (ii), $av-(a-2)w$ has degree $2$, so Riemann--Roch on the genus-$2$ graph $B$ shows that it is equivalent to an effective divisor.

For (iii), choose positive integers $m_0,m_1,m_2$ with $m_0+m_1+m_2=a$. The conditions $m_0\ell_0=m_1\ell_1=m_2\ell_2$ are two independent conditions and give $av\sim aw$. For the second statement, impose only $m_0\ell_0=m_1\ell_1$. Choosing $\ell_2$ in a suitable nonempty open interval produces a unique breakpoint $x$ in the third edge at which the slope drops by $1$, giving $av\sim(a-1)w+x$.
\end{proof}

\begin{Prop}\label{prop:gonch2-ban}
Let $\Gamma$ be a metric banana path with $N$ genus-$2$ bananas.
\begin{enumerate}[label=(\roman*)]
\item If $\Gamma$ is generic, then $\operatorname{gon}(\Gamma)=N+1$, which is the gonality of a general curve of genus $2N$.
\item For every $d$ with $3\leq d\leq N+1$, there is a locus of codimension $-\rho(2N,d,1)=2(N+1-d)$ consisting of $d$-gonal metric banana paths of this combinatorial type.
\end{enumerate}
\end{Prop}

\begin{proof}
For (i), the general upper bound for the gonality of a metric graph of genus $2N$ gives $\operatorname{gon}(\Gamma)\leq N+1$.

Suppose that $D$ has rank at least $1$ and degree $d\leq N$. For each $i$, let $D_{v_i}$ be the $v_i$-reduced divisor equivalent to $D$ and set $a_i=D_{v_i}(v_i)$. Then $a_i\geq1$.

Fix $1\leq i\leq N$. By Lemma~\ref{lem:agreement-on-components-general}, the restrictions of $D_{v_{i-1}}$ and $D_{v_i}$ outside $B_i$ agree. Let $A,A'$ be their remaining restrictions to $B_i$. Then $A\sim A'$. Put $c=\deg A=\deg A'$, $k_i=c-a_{i-1}$, and $\delta_i=c-a_i$.

We claim that $k_i+\delta_i\geq2$. If $k_i+\delta_i=0$, then $cv_{i-1}\sim cv_i$, contrary to Lemma~\ref{LSBanana}(i). If $k_i+\delta_i=1$, then, after possibly interchanging $v_{i-1}$ and $v_i$, we have $cv_{i-1}\sim(c-1)v_i+x$ for some $x\in B_i$. If $x=v_{i-1}$, cancelling it gives $(c-1)v_{i-1}\sim(c-1)v_i$; otherwise Lemma~\ref{LSBanana}(i) applies directly. In either case we obtain a contradiction.

The quantities $k_i$ count the degree lying in $B_i$ away from $v_{i-1}$ in $D_{v_{i-1}}$, while the $\delta_i$ do the same away from $v_i$ in $D_{v_i}$. Agreement on components therefore gives $\sum_{i=1}^N k_i\leq d-a_0\leq d-1$ and $\sum_{i=1}^N\delta_i\leq d-a_N\leq d-1$.
Consequently, $2N\leq\sum_{i=1}^N(k_i+\delta_i)\leq 2d-2$,
contradicting $d\leq N$. Hence $\operatorname{gon}(\Gamma)=N+1$.

For (ii), if $d=N+1$, the general upper bound shows that every metric of this combinatorial type is $d$-gonal, giving codimension $0=-\rho(2N,d,1)$. Assume $3\leq d\leq N$ and write $k=\lfloor d/2\rfloor$.

The construction uses the three transitions of Lemma~\ref{LSBanana}: losing $2$ from the coefficient at the next vertex imposes no condition, losing $1$ imposes one condition, and losing $0$ imposes two conditions.

If $d=2k+1$, impose the loss-$0$ condition $d v_{i-1}\sim d v_i$ on $B_{k+1},\dots,B_{N-k}$ and no conditions on the remaining bananas. Starting from $d v_k$, use loss-$2$ transitions through the $k$ bananas on either end. Since $d-2k=1$, this gives effective representatives containing every $v_i$.

If $d=2k$, impose loss $1$ on $B_k$ and $B_{N-k+1}$, loss $0$ on $B_{k+1},\dots,B_{N-k}$, and no conditions on the remaining bananas. Starting from $d v_k$, the first transition on either side leaves coefficient $d-1=2k-1$, after which $k-1$ loss-$2$ transitions again give effective representatives containing every $v_i$.

Since $V_{-}(\Gamma)=\{v_0,\dots,v_N\}$ is rank-determining, the divisor $d v_k$ has rank at least $1$ in either case. If $d=2k+1$, there are $N-2k=N+1-d$ loss-$0$ bananas, so the codimension is $2(N+1-d)$. If $d=2k$, there are $N-2k$ loss-$0$ bananas and two loss-$1$ bananas, so the codimension is $2(N-2k)+2=2(N+1-d)$. Since $\rho(2N,d,1)=2d-2N-2$, this is exactly $-\rho(2N,d,1)$.
\end{proof}

The same methods extend, \emph{mutatis mutandis}, to metric banana paths whose bananas have genus one or two, although we do not pursue this generalization here.


\end{document}

%% file: finitization-figure.tex
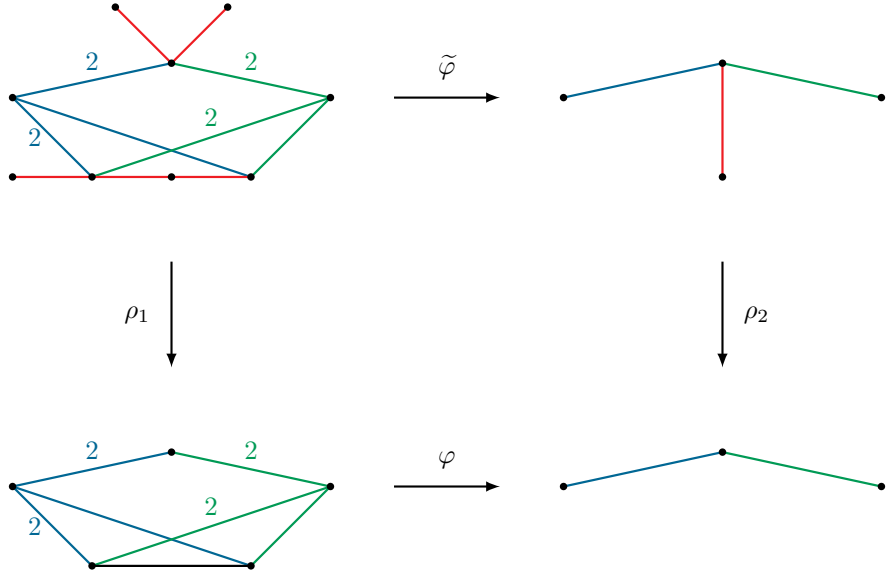
\begin{figure}[H]
\centering
\begin{tikzpicture}[
    scale=0.7,
    every node/.style={font=\small},
    >=latex
]

\begin{scope}[shift={(-5.2,3.5)}]

\coordinate (u) at (-1.5,0);
\coordinate (v) at (1.5,0);
\coordinate (x) at (0,2.15);
\coordinate (p) at (-3.0,1.5);
\coordinate (q) at (3.0,1.5);
\coordinate (w) at (0,0);

\coordinate (un)  at (-3,0);
\coordinate (xn1) at (-1.06,3.21);
\coordinate (xn2) at (1.06,3.21);

\draw[MidnightBlue,line width=\edgewidth]
    (u)--(p)
    node[midway,left] {$2$};

\draw[MidnightBlue,line width=\edgewidth]
    (v)--(p);

\draw[MidnightBlue,line width=\edgewidth]
    (x)--(p)
    node[midway,above] {$2$};

\draw[ForestGreen,line width=\edgewidth]
    (u)--(q)
    node[midway,above] {$2$};

\draw[ForestGreen,line width=\edgewidth]
    (v)--(q);

\draw[ForestGreen,line width=\edgewidth]
    (x)--(q)
    node[midway,above] {$2$};

\draw[Red,line width=\edgewidth] (u)--(w);
\draw[Red,line width=\edgewidth] (v)--(w);

\draw[Red,line width=\edgewidth] (u)--(un);

\draw[Red,line width=\edgewidth] (x)--(xn1);
\draw[Red,line width=\edgewidth] (x)--(xn2);

\vertex{u}
\vertex{v}
\vertex{x}
\vertex{p}
\vertex{q}
\vertex{w}
\vertex{un}
\vertex{xn1}
\vertex{xn2}
\end{scope}

\begin{scope}[shift={(5.2,3.5)}]

\coordinate (y)  at (0,2.15);
\coordinate (pp) at (-3,1.5);
\coordinate (qp) at (3,1.5);
\coordinate (r)  at (0,0);

\draw[MidnightBlue,line width=\edgewidth] (y)--(pp);
\draw[ForestGreen,line width=\edgewidth]   (y)--(qp);
\draw[Red,line width=\edgewidth]  (y)--(r);

\vertex{y}
\vertex{pp}
\vertex{qp}
\vertex{r}
\end{scope}

\begin{scope}[shift={(-5.2,-3.85)}]

\coordinate (u) at (-1.5,0);
\coordinate (v) at (1.5,0);
\coordinate (x) at (0,2.15);
\coordinate (p) at (-3.0,1.5);
\coordinate (q) at (3.0,1.5);

\draw[MidnightBlue,line width=\edgewidth]
    (u)--(p)
    node[midway, left] {$2$};

\draw[MidnightBlue,line width=\edgewidth]
    (v)--(p);

\draw[MidnightBlue,line width=\edgewidth]
    (x)--(p)
    node[midway,above] {$2$};

\draw[ForestGreen,line width=\edgewidth]
    (u)--(q)
    node[midway,above] {$2$};

\draw[ForestGreen,line width=\edgewidth]
    (v)--(q);

\draw[ForestGreen,line width=\edgewidth]
    (x)--(q)
    node[midway,above] {$2$};

\draw[black,line width=\edgewidth]
    (u)--(v);

\vertex{u}
\vertex{v}
\vertex{x}
\vertex{p}
\vertex{q}
\end{scope}

\begin{scope}[shift={(5.2,-3.85)}]

\coordinate (y)  at (0,2.15);
\coordinate (pp) at (-3,1.5);
\coordinate (qp) at (3,1.5);

\draw[MidnightBlue,line width=\edgewidth] (y)--(pp);
\draw[ForestGreen,line width=\edgewidth]   (y)--(qp);

\vertex{y}
\vertex{pp}
\vertex{qp}
\end{scope}


\draw[->,thick]
    (-1,5) -- (1,5)
    node[midway,above=3pt] {$\widetilde{\varphi}$};

\draw[->,thick]
    (-1,-2.35) -- (1,-2.35)
    node[midway,above=3pt] {$\varphi$};


\draw[->,thick]
    (-5.2,1.9) -- (-5.2,-0.1)
    node[midway,left=4pt] {$\rho_1$};

\draw[->,thick]
    (5.2,1.9) -- (5.2,-0.1)
    node[midway,right=4pt] {$\rho_2$};

\end{tikzpicture}

\caption{A finitization of a non-degenerate harmonic morphism.}
\label{fig:finitization-square}
\end{figure}